\documentclass[12pt]{article}
\usepackage{amssymb}
\usepackage{amsmath}
\usepackage{amsfonts}
\usepackage{amsthm}
\usepackage{amscd}
\usepackage{delarray}
\numberwithin{equation}{section}
\newtheorem{thm}{Theorem}[section]
\newtheorem{cor}[thm]{Corollary}
\newtheorem{lem}[thm]{Lemma}
\newtheorem{prop}[thm]{Proposition}

\newtheorem{fact}[thm]{Fact}

\newtheorem{thmalph}{Theorem}

\newtheorem{coralph}[thmalph]{Corollary}

\theoremstyle{definition}
\newtheorem{definition}[thm]{Definition}

\newtheorem{remark}[thm]{Remark}
\newtheorem{exm}[thm]{Example}

\newenvironment{newenumerate}{
 \begin{enumerate}
 
 }{\end{enumerate}}

\newfont{\bg}{cmr10 scaled \magstep4}

\newcommand{\bigzerou}{\smash{\lower1.7ex\hbox{\bg 0}}}
\newcommand{\R}{\mathbb{R}}
\newcommand{\C}{\mathbb{C}}

\newcommand{\g}{\mathfrak{g}}

\newcommand{\sL}{\mathfrak{sl}}

\newfont{\Bg}{
 cmr10 scaled\magstep5}

\def \c{\mathfrak{c}}

\def \a{\mathfrak{a}}
\def \k{\mathfrak{k}}
\def \p{\mathfrak{p}}
\def \h{\mathfrak{h}}

\def \t{\mathfrak{t}}

\newcommand{\PR}{\operatorname{pr}}

\def \ad{\operatorname{ad}}

\def \Ad{\operatorname{Ad}}

\def \GL{\operatorname{GL}}

\def \Sp{\operatorname{Sp}}

\def\rarrowsim{\smash{\mathop{\,\longrightarrow\,}\limits
  ^{\lower1.5pt\hbox{$\scriptstyle\sim$}}}} %TK
\newcommand{\bysame}{%
\leavevmode\hbox to 3em{\hrulefill}\,}

\makeatletter
\renewcommand{\@makecaption}[2]{%
  \vskip\abovecaptionskip
  \sbox\@tempboxa{#1 #2}%
  \ifdim \wd\@tempboxa >\hsize
    #1: #2\par
  \else
    \global \@minipagefalse
    \hb@xt@\hsize{\hfil\box\@tempboxa\hfil}%
  \fi
  \vskip\belowcaptionskip}
\makeatother

\title{Geometry of coadjoint orbits and multiplicity-one branching laws for 
symmetric pairs
\\[2ex]
\textit{\normalsize
Dedicated to Alexandre Kirillov
 on the occasion of his 81st birthday
 with admiration 
}}
\author{Toshiyuki Kobayashi
\\
\normalsize{Graduate School of Mathematical Sciences,
 and Kavli IPMU,}
\\
 and Salma Nasrin
\\
\normalsize{Department of Mathematics, University of Dhaka}}
\date{}

\begin{document}

\maketitle

\begin{abstract}
Consider the restriction of an irreducible unitary representation $\pi$ of a 
Lie group $G$ to its subgroup $H$. 
Kirillov's revolutionary idea on the orbit method
 suggests
 that the multiplicity of an irreducible $H$-module $\nu$
 occurring in the restriction $\pi|_H$ could be read from 
 the coadjoint action of $H$
 on $\mathcal{O}^G \cap \PR^{-1}({\mathcal{O}}^H)$, 
 provided $\pi$ and $\nu$ are \lq{geometric quantizations}\rq\
 of a $G$-coadjoint orbit $\mathcal{O}^G$ 
 and an $H$-coadjoint orbit $\mathcal{O}^H$, 
respectively,
 where $\PR \colon \sqrt{-1}\g^{\ast} \to \sqrt{-1}\h^{\ast}$ is 
 the projection dual to the inclusion $\h \subset \g$ of Lie algebras.  
Such results were previously established by Kirillov, Corwin and Greenleaf 
for nilpotent Lie groups.  
%In contrast, not much have been known for semisimple Lie groups
% for which the orbit method is more involved.  

\smallskip
In this article, we highlight specific elliptic orbits $\mathcal{O}^G$
 of a semisimple Lie group $G$
 corresponding to highest weight modules
 of scalar type. 
We prove that the Corwin--Greenleaf number 
 $\sharp(\mathcal{O}^G \cap \PR^{-1}({\mathcal{O}}^H))/H$
is either zero or one for any $H$-coadjoint orbit $\mathcal{O}^H$, 
 whenever $(G,H)$ is a symmetric pair of holomorphic 
type. 
Furthermore, we determine the coadjoint orbits $\mathcal{O}^H$ 
 with nonzero Corwin--Greenleaf number.   
Our results coincide with the prediction of the orbit philosophy, 
 and can be seen as \lq{classical limits}\rq\
 of the 
{\it multiplicity-free} branching laws 
 of holomorphic discrete series representations 
(T.~Kobayashi [Progr.~Math.~2007]). 
\end{abstract}

%         also 32M10, 51M35]


\section{Introduction} \label{S:intro}

The Kirillov--Kostant--Duflo orbit philosophy bridges the unitary dual 
$\widehat{G}$ of a Lie group $G$ and the set $\sqrt{-1}\g^*/G$ 
of coadjoint orbits. 
The orbit method works perfectly
 for simply 
connected nilpotent Lie groups and certain solvable groups $G$, 
including the one-to-one correspondence 
({\it{Kirillov correspondence}})
between $\widehat{G}$ and 
$\sqrt{-1}\g^* / G$ and the functorial properties for inductions 
and restrictions (see \cite{corwin, fujiwara, kirillov62, kirillov99, kirillov}
 and the references therein). 

\smallskip
Our interest is in the restriction of representations to subgroups 
 and its counterpart in the geometry
 of coadjoint orbits.  
First, 
 we consider the representation-theoretic side.  
Let $\pi$ be an irreducible unitary representation of a Lie group $G$,
 and $H$ a subgroup of $G$. Then the restriction $\pi |_{H}$ is 
decomposed into the direct integral of irreducible unitary 
representations of $H$: 

\begin{equation} \label{E:branch} 
\pi |_{H} \simeq \int_{\widehat H}^{\oplus} m_{\pi} (\nu) \nu d 
\mu (\nu).  
\end{equation}
Here $\mu$ is a Borel measure on the unitary dual $\widehat H$, 
 and the measurable function 
$m_{\pi}\colon \widehat{H} \rightarrow \mathbb{N} \cup \{\infty \}$ 
 stands for the multiplicity.  
The decomposition \eqref{E:branch} is called {\it branching law}
 of the restriction $\pi |_{H}$. 
The irreducible decomposition \eqref{E:branch} is unique
 up to equivalence
 if $H$ is of type I, 
 {\it{e.g.}}, 
 if $H$ is nilpotent or reductive.  
We denote by ${\operatorname{Supp}}_H(\pi |_{H})$
 the subset of $\widehat H$ 
 that is the support of the direct integral \eqref{E:branch}.

\smallskip 
Next, 
we consider the coadjoint orbit side.  
The {\it Corwin--Greenleaf multiplicity function} 
\[
n\colon (\sqrt{-1} \g^*/G) \times (\sqrt{-1} \h^*/H) \to \mathbb N \cup \{\infty\}, 
\]
counts the number of $H$-orbits
 in the intersection $\mathcal{O}^G \cap \PR^{-1} (\mathcal{O}^H)$,  
 namely, 
\begin{equation} \label{E:cgmf}
n(\mathcal{O}^G, \mathcal{O}^H) := \sharp \biggl( \bigl(
\mathcal{O}^G \cap \PR^{-1} (\mathcal{O}^H) \bigr)/H \biggr),
\end{equation}
where $\PR\colon \sqrt{-1} \g^* \rightarrow \sqrt{-1} \h^* $ 
is the natural projection. 
When $G$ is a simply connected nilpotent Lie group 
 and $H$ is a connected subgroup, 
 Corwin and Greenleaf \cite{corwin} proved
 that the multiplicity $m_{\pi} (\nu)$ coincides with the geometric number 
$n(\mathcal{O}_{\pi}^G, \mathcal{O}_{\nu}^H)$
 almost everywhere if $\mathcal{O}_{\pi}^G \subset \sqrt{-1} \g^*$
 and $\mathcal{O}_{\nu}^H \subset \sqrt{-1} \h^*$
 are the coadjoint orbits corresponding to $\pi \in \widehat{G}$ and 
$\nu \in \widehat{H}$, respectively, under the Kirillov correspondence.  
Thus the result \cite{corwin} is summarized as 
\begin{equation*}
\begin{array}{ccc}
\text{representation theory}
& & 
\text{geometry of coadjoint orbits}
\\
m_{\pi} (\nu) \quad 
&=& \quad n(\mathcal{O}_{\pi}^G, \mathcal{O}_{\nu}^H).  
\end{array}
\end{equation*}

\smallskip
In contrast to nilpotent groups, it has been 
observed by many specialists
 that the orbit philosophy does not work very well for noncompact semisimple Lie groups $G$, 
 see {\it{e.g.}}, \cite{kirillov99, kirillov, kobayashi94*}. 
Indeed, 
 there does not exist a reasonable one-to-one 
correspondence between $\widehat G$ and $\sqrt{-1}\g^*/G$:
\lq{missing}\rq\
 of coadjoint orbits corresponding to complementary series representations
 ({\textit{cf}}. \cite[Thm.~2.30]{xHKM}), 
 missing of some \lq{unipotent representations}\rq\
 that are supposed to be attached to nilpotent coadjoint orbits, 
 and failure of irreducibility or vanishing
 of nontempered Vogan--Zuckerman $A_{\mathfrak{q}}(\lambda)$-modules
 that are supposed to be attached to elliptic coadjoint orbits
 even for \lq{positive}\rq\
 $\lambda$
 (\cite{K92, Vog88})
 among others, 
 and consequently a rigorous formulation
 for \lq{functional properties}\rq\ 
 in the orbit method is not obvious.  
Nevertheless, 
 we still expect that Kirillov's orbit philosophy 
 provides useful information
 and new insights on unitary representation theory and the geometry of 
coadjoint orbits. 
In fact, 
some successful cases about the functorial properties of the orbit method
 for discretely decomposable restrictions
 to noncompact reductive subgroups $H$ 
 include Kobayashi--{\O}rsted \cite{k-o98}
 for minimal representations
 attached to minimal coadjoint orbits, 
 Duflo--Vargas \cite{duflovargas}
  for discrete series representations
 attached to strongly elliptic orbits,
and a recent work by Paradan \cite{xParadan2015}
for holomorphic discrete series representations.

\smallskip
In this article, we consider the case where $(G, H)$ is a symmetric 
pair of holomorphic type
 (see Definition \ref{def:GHholo} below).  
A typical example is $(G, H) =(\Sp(n, \R), {\operatorname{U}}(p, q))$ and 
$(\Sp(n, \R), \Sp(p, \R) \times \Sp(q, \R))$ with $p+q = n$. 
We highlight on specific coadjoint orbits $\mathcal{O}^G$
 (see \eqref{E:OGZ}), and find 
explicitly the Corwin--Greenleaf function
 for an arbitrary coadjoint orbit $\mathcal{O}^H$
 for any symmetric pair $(G, H)$ of holomorphic type.

Our main results are Theorems \ref{thm:A} and \ref{T:elliptic} 
 which are predicted by the orbit philosophy
 as the `classical limit' of multiplicity-free 
 discretely decomposable restrictions 
 of unitary representations
 that were established earlier 
 (see \cite{kobayashi97, kobayashi07}).  
Our results can be interpreted also from the viewpoint of symplectic geometry, 
 namely,
 the momentum map $\mu \colon \mathcal{O}^G \to \sqrt{-1}\h^{\ast}$
 for the Hamiltonian action
 of the subgroup $H$ on ${\mathcal{O}}^G$
 endowed with the Kirillov--Kostant--Souriau symplectic form
 is a proper map (Corollary \ref{cor:proper})
 with explicit description of its image (Theorem \ref{thm:Cprime})
 indicating that the geometric quantization ${\mathcal{Q}}$
 commutes with reduction
 in this setting,
 symbolically written as 
\[
  {\mathcal{Q}}^H \circ \mu =\text{ Restriction} \circ {\mathcal{Q}}^G. 
\] 

Thus the main features are summarized as follows.  
\begin{equation*}
\begin{array}{l|l}
\text{coadjoint orbits}
&\text{unitary representations}
\\
\hline
\\
\vspace {-0.9cm}
&
\\
\text{${\mathcal{O}}^G$}
&\quad \text{$\pi^G$}
\\
{\mathcal{O}}^G \cap \sqrt{-1} ([\k,\k] + \p)^{\perp} \ne \{0\}
&\quad \text{holomorphic rep.~of scalar type}
\\
\text{$n({\mathcal{O}}^G, {\mathcal{O}}^H) \le 1$ \quad ($\forall {\mathcal{O}}^H$)
\hphantom{MMm} Thm.~\ref{thm:A}}
&\quad\text{$\pi^G|_H$ is multiplicity-free \hphantom{MMMm}\cite{kobayashi97}}
\\
\text{$\mu \colon {\mathcal{O}}^G \to \sqrt{-1}\h^{\ast}$ is proper
\hphantom{MMi}Cor.~\ref{cor:proper}}
&\quad\text{$\pi^G|_H$ is discretely decomposable \cite{xkAnn98}}
\\
\text{$\mathcal O^G \cap \operatorname{pr}^{-1}(\mathcal{O}_\nu^H)\neq \emptyset$
\hphantom{MMMMMi} Thm.~\ref{T:elliptic}}
&\quad\text{${\operatorname{Hom}}_H(\pi_{\nu}^H, \pi^G|_H) \ne \{0\}$ 
\hphantom{MMMM}\cite{kobayashi97}}
\\
\text{${\operatorname{Image}} (\mu \colon {\mathcal{O}}^G \to \sqrt{-1}\h^{\ast})$
\hphantom{MMM} Thm.~\ref{thm:Cprime}}
&\quad\text{$\operatorname{Supp}_H(\pi^G|_H) \, (\subset \widehat{H})$ 
\hphantom{MMMMMb}\cite{kobayashi97}}
\end{array}
\end{equation*}

Theorems \ref{thm:A}, \ref{T:elliptic} and \ref{thm:Cprime}
 for $H=K$ (maximal compact subgroups) were proved
 in the Ph.~D.~thesis \cite{xNasrin2003} of S.~Nasrin
 at The University of Tokyo in 2003, 
 see also \cite{xNasrin2008, nasrin}. 
Alternatively,
 Theorems \ref{thm:A} and \ref{T:elliptic} for $H=K$ follow from 
a result of McDuff \cite{xMcDuff1988}, extended by Deltour \cite{xDeltour2013}
that the coadjoint orbit ${\mathcal{O}}^G$ is symplectomorphic 
 to the vector space $\p$ and from Paradan \cite[Prop.~5.5]{xParadan2008},  
 too. 

The results of this article for noncompact $H$ were delivered at
the workshop \lq\lq Geometric Quantization in the Non-compact Setting\rq\rq\ organized
by L.~Jeffrey, X.~Ma and M.~Vergne
at Oberwolfach, Germany, 13--19 February 2011, and were collected in 
\cite{kobayashi11}. 
Theorem~A was announced earlier in \cite{k-n03}.

\medskip 
\section{Statement of main results}

In this section 
 we formulate our main results
 on the geometry of coadjoint orbits
 that are predicted by the representation-theoretic results
 via the orbit method.  
Theorem A is the counterpart
 of the multiplicity-freeness property
 of the restriction $\pi|_H$ (Fact \ref{F:mf}), 
 and Theorem C is that of its explicit branching law
 (Fact \ref{F:f1}).  

\subsection{Orbit geometry for multiplicity-free representations}

Let $G$ be a noncompact, simple Lie group, 
 $K$ a maximal compact subgroup modulo
 the center of $G$, 
 $\theta$ the corresponding Cartan involution of $G$, 
 and $\g = \k + \p$ for the Cartan decomposition
 of the Lie algebra $\g$ of $G$.  
We say $G$ is a {\it{Hermitian Lie group}}
 if the associated Riemannian symmetric space $G/K$ is 
 a Hermitian symmetric space,
 or equivalently,
 if the center $\c(\k)$ of $\k$ is nonzero.  
In this case $\c(\k)$ is actually one-dimensional.

\smallskip
An irreducible representation $\pi$ of $G$
 is said to be a lowest weight module
 if its underlying $(\g,K)$-module
 is ${\mathfrak{b}}$-finite
 for some Borel subalgebra ${\mathfrak{b}}$
 of the complexified Lie algebra 
 $\g_{\mathbb{C}}=\g \otimes_{\mathbb{R}} {\mathbb{C}}$.  
Moreover it is said to be {\it of scalar} type
 if the minimal $K$-type of $\pi$ is one-dimensional.  
There exists an infinite-dimensional irreducible lowest weight representation
 of a simple Lie group $G$
 if and only if $G$ is a Hermitian Lie group.  
For any simply-connected Hermitian group $G$, 
there exist continuously many lowest weight modules of $G$ of scalar type.

Geometrically,
 any irreducible lowest weight representation $\pi$
 of scalar type
 can be realized
 in the space of holomorphic sections
 for a $G$-equivariant holomorphic line bundle
 over the Hermitian symmetric space $G/K$. 
This geometric observation is brought into the orbit method:
such $\pi$ can be seen
 as a \lq{geometric quantization}\rq\
 of a coadjoint orbit ${\mathcal{O}}^G$ satisfying
\begin{equation} \label{E:OGZ} 
\mathcal{O}^G  \cap \sqrt{-1} ([\k, \k] + \p)^{\perp} \ne \emptyset. 
\end{equation} 
%(see \cite[Example 2.7]{kobayashi94*}). 
We note that $([\k, \k] + \p)^{\perp} \ne \{0\}$ 
if and only if $G$ is of Hermitian type.
In this case the coadjoint orbit
 ${\mathcal{O}}^G$ satisfying \eqref{E:OGZ} is isomorphic 
 to the Hermitian symmetric space $G/K$
 as $G$-spaces
 unless ${\mathcal{O}}^G = \{0\}$.

\smallskip
Let $ \tau $ be an involutive automorphism of $G$. 
We say that $(G,H)$ is a {\it{symmetric pair}}
 if $H$ is an open subgroup of the 
fixed point group $G^{\tau} := \{g \in G : \tau g = g \}$. 
In this article, we shall assume $H$ is connected for simplicity.

\smallskip
For the representation theory side,
 we recall the following multiplicity-free theorem: 

\begin{fact}
[{\cite{kobayashi07}}]
\label{F:mf}
For any  irreducible unitary lowest weight representation $\pi$ of scalar 
type of $G$ and for any symmetric pair $(G, H)$, the 
restriction $\pi |_{H}$ is multiplicity-free. 
\end{fact} 

See \cite[Thm.~A]{kobayashi97}
 for the proof based on the theory 
 of \lq{visible actions}\rq\
 on complex manifolds.    

\smallskip
Suppose $\tau$ is an involutive automorphism of Hermitian Lie group $G$ 
such that $\tau \theta = \theta \tau$. Then, $\tau$ stabilizes the center 
$\c(\k)$ of $\k$. Since $\dim \c(\k) = 1$ for a Hermitian Lie group $G$, 
$\tau|_{\c(\k)}$ is either 
$\operatorname{id}$ or $- \operatorname{id}$. 
On the other hand, since $\tau(K) = K$, $\tau$ also acts on 
$G/K$ as a diffeomorphism. 
This action is either holomorphic or anti-holomorphic 
according to $\tau|_{\c(\k)} = \operatorname{id}$ or 
$- \operatorname{id}$. 
\begin{definition}
\label{def:GHholo}
The involution $\tau$ (or the corresponding symmetric pair $(G, H)$) is 
said to be {\it holomorphic} or {\it anti-holomorphic}, 
 if $\tau|_{\c(\k)}={\operatorname{id}}$
 or $-{\operatorname{id}}$, 
respectively. 
\end{definition}
The Cartan involution $\theta$ is always of holomorphic type. 

\begin{exm}
Let $G = \Sp(n, \R)$. 
Then the pair $(G,H)$ is of holomorphic type if 
$H = \operatorname{U}(p,q)$ or  
$\Sp(p, \R) \times \Sp(q, \R) \quad (p+q =n)$, 
whereas it is  
of anti-holomorphic type if $H = \GL(n, \R))$. 
See \cite[Tables 3.4.1 and 3.4.2]{kobayashi07}
 for the list of all the irreducible symmetric pairs $(G,H)$
 on the Lie algebra level 
 that are of holomorphic and anti-holomorphic types. 
\end{exm}

In this article, 
 we shall treat the case where $\tau$ is of holomorphic type. 
This implies that the branching law \eqref{E:branch}
 of the restriction $\pi |_{H}$ does not contain
 any continuous spectrum, 
 and is discretely decomposable 
for any lowest weight module $\pi$ (\cite{kobayashi94,xkAnn98}). 

\smallskip
The first main result of this article is to give the counterpart of Fact 
\ref{F:mf} in terms of the Corwin--Greenleaf function \eqref{E:cgmf}
 in the geometry of coadjoint orbits.

\smallskip 
\begin{thmalph}
\label{thm:A}
Let $G$ be a Hermitian Lie group,
 $(G,H)$ a symmetric pair
 of holomorphic type,
 and $\mathcal{O}^{G}$ a coadjoint orbit
 in $\sqrt{-1} \g^{\ast}$
 of $G$ satisfying \eqref{E:OGZ}. Then 
$$ 
n(\mathcal{O}^{G}, \mathcal O^{H}) \le 1
$$
for any coadjoint orbit $\mathcal O^{H}$ in $\sqrt{-1}\h^*$.
\end{thmalph}

\smallskip 
It should be noted
 that the Corwin--Greenleaf function $n ( \mathcal{O}^G, \, \mathcal{O}^H)$ may be infinite in general even for 
a symmetric pair $(G, H)$ if we drop the assumption \eqref{E:OGZ} 
(see \cite{k-n03} for such an example).  

Since $H$ is connected, Theorem A implies the following 
topological result: 

\noindent 
\begin{coralph}
The intersection 
$\mathcal{O}^G \cap \PR^{-1} (\mathcal{O}^H)$ is connected for any 
coadjoint orbit $\mathcal{O}^G$ in $\sqrt{-1} \g^*$ satisfying \eqref{E:OGZ} 
and for any coadjoint orbit $\mathcal{O}^H$ in $ \sqrt{-1} \h^*$. 
\end{coralph}

\smallskip
In the special case $H=K$, 
the connectedness of the intersection
 $\mathcal{O}^G \cap \PR^{-1} (\mathcal{O}^H)$  
could be derived also from a general theorem 
concerning Hamiltonian actions of connected compact group on symplectic 
manifolds with proper moment maps (see \cite{xParadan2015, sj}). 
%\smallskip t
On the other hand, 
$\mathcal{O}^G \cap \PR^{-1} (\mathcal{O}^H)$ can be disconnected 
when the momentum map $\mu \colon \mathcal{O}^G \to \sqrt{-1}\h^{\ast}$
 is not proper
 (see \cite[Fig.~4.6]{k-n03} for such an example). 

\subsection{Nonvanishing condition 
 for the Corwin--Greenleaf function} \label{SS:s2} 

The second main result of this article is
 a necessary and sufficient condition 
 for the Corwin--Greenleaf function 
$n ( \mathcal{O}^G, \, \mathcal{O}^H)$
 to be nonzero. 
In order to establish it, we need to 
fix a parametrization of $\mathcal{O}^G$ and $\mathcal{O}^H$. 

\smallskip
Suppose $G$ is a simple Lie group of Hermitian type.  
Then the center ${\mathfrak{c}}(\k)$ of $\k$
 is one-dimensional,
 and there exists a characteristic 
element $Z \in \sqrt{-1}\c(\k)$ such that 
\begin{equation} \label{E:kp}
\g_{\C} := \g \otimes_{\R} \C = \k_{\C} + \p_{+} + \p_{-} 
\end{equation}
is the direct sum decomposition of the eigenspaces
 of $\ad (Z)$ with eigenvalues 
$0$, $+ 1$ and $-1$, respectively.
Then $G/K$ carries a $G$-invariant complex structure
 with holomorphic tangent bundle
 $G \times_K \p_+ \to G/K$.  
 
Suppose $\tau$ is an involutive automorphism of $G$ commuting with the Cartan 
involution $\theta$. 
We use the same letters $\tau$, $\theta$ to denote the 
complex linear extensions of their differentials. 
We take a maximal abelian subspace of $\k^{\tau} = \h \cap \k$, 
 and extend it to a maximal abelian subspace $\t$ of $\k$. 
Then $\t^{\tau}=\h \cap \t$ is a Cartan subalgebra
 of $\k^{\tau}$.  
Let $\Delta(\k,\t)$ ($\subset \sqrt{-1} \t^{\ast}$)
 and $\Delta(\k^{\tau},\t^{\tau})$ ($\subset \sqrt{-1} (\t^{\tau})^{\ast}$)
 be the root systems
 of the pair $(\k_{\mathbb{C}}, \t_{\mathbb{C}})$
 and $(\k_{\mathbb{C}}^{\tau}, \t_{\mathbb{C}}^{\tau})$, 
 respectively.  
Then we can 
choose compatible positive systems 
$\Delta^{+} (\k, \, \t)$ and $\Delta^{+} (\k^{\tau}, \, \t^{\tau})$ 
in the sense that 
\begin{equation}\label{E:compa}
\alpha |_{\t^{\tau}} \in \Delta^{+} (\k^{\tau}, \, \t^{\tau}) \quad 
\text{for any } \alpha \in \Delta^{+} (\k, \, \t).
\end{equation} 

We write $\sqrt{-1} (\t^{*})_{+}$ for the dominant Weyl chamber
 with respect to the positive system $\Delta^{+} (\k, \, \t)$, and 
$\sqrt{-1} (\t^{\tau})_{+}^*$ to $\Delta^{+} (\k^{\tau}, \, \t^{\tau})$.

Hereafter we assume
 that $\tau$ is of holomorphic type (Definition \ref{def:GHholo}). 
Since $\tau Z = Z$, the direct sum 
decomposition \eqref{E:kp} is stable under $\tau$. 
Thus we have a direct sum decomposition 
$$\p_{+} = \p_{+}^{\tau} + \p_{+}^{-\tau},$$ 
where we set
$$\p_{+}^{\pm\tau} := \{ X \in \p_{+} : \tau X = \pm X \, \}.$$ 

For a $\t^{\tau}$-stable subspace $F$ in $\p_+$, 
 let $\Delta(F)$ denote the set of weights of $F$ with 
respect to $\t^{\tau}$. It is a finite set in $\sqrt{-1} (\t^{\tau})^{*}$.

\smallskip
The subgroup $G^{\tau\theta}$ is locally isomorphic to the direct product
 of a compact normal subgroup $G^{(0)}$
 and noncompact simple Lie subgroups $G^{(i)}$
 ($1 \le i \le L$).  
Correspondingly,
 the Lie algebra $\g^{\tau\theta}$ is decomposed into the direct sum:
\begin{equation}
\label{eqn:Hdeco}
   \g^{\tau\theta} 
  \simeq \g^{(0)} \oplus \g^{(1)} \oplus \cdots \oplus \g^{(L)}.  
\end{equation}
\begin{remark}
\label{rem:L}
By the classification \cite[Table 3.4.1]{kobayashi07}, 
 we see 
\[
  L = 1\quad\text{or}\quad 2.  
\]
For example,
 when $G=\Sp(p+q,\R)$, 
 $L=1$
 if $H={\operatorname{U}}(p,q)$
 and $L=2$
 if $H=\Sp(p,\R) \times \Sp(q,\R)$.  
\end{remark}

Two roots $\alpha$ and $\beta$ are called {\it strongly orthogonal} if neither 
$\alpha + \beta$ nor $\alpha - \beta$ is a root. 
For each $i$ ($1 \le i \le L$), 
 we denote by $r_i$ ($\ge 1$)
 the real rank of $\g^{(i)}$, 
 and take a maximal set of 
strongly orthogonal roots 
$\{\nu_1^{(i)}, \, \cdots, \, \nu_{r_i}^{(i)}\}$ in 
 $\Delta(\p_{+} \cap \g^{(i)})$ $(\subset \sqrt{-1}(\t^{\tau})^{\ast})$
 such that 
\begin{enumerate}
\item $\nu_1^{(i)}$ is the highest in $\Delta(\p_{+}\cap \g^{(i)})$. 
\item $\nu_k^{(i)}$ is the highest in the set of all $\nu$ in 
$\Delta(\p_{+} \cap \g^{(i)})$ such that $\nu$ is strongly orthogonal to 
$\nu_1^{(i)}, \cdots, \nu_{k-1}^{(i)} \; (2 \le k \le r_i)$.
\end{enumerate} 

We note that the split rank of the semisimple 
symmetric space $G/H$ equals the real rank of $G^{\tau\theta}$, 
 which is given by 
 $r:=r_1 + \cdots + r_L$. 

\medskip
We set, 
 for $1 \le i \le L$, 
\begin{align*}
   C_+^{(i)}&:=\{(t_j^{(i)})_{1 \le j \le r_i}\in {\mathbb{R}}^{r_i}
                :
                t_1^{(i)} \ge \cdots \ge t_{r_i}^{(i)} \ge 0\}, 
\\
   \Lambda^{(i)}&:= C_+^{(i)} \cap {\mathbb{Z}}^{r_i}, 
\end{align*}
 and define a closed convex cone in $\sqrt{-1}(\t^{\tau})^*$ by 
\begin{equation} \label{E:cone}
   \operatorname{Cone} (\p_{+}^{-\tau}) 
:= \{ \sum_{i=1}^L \sum_{j=1}^{r_i} t_j^{(i)} \nu_j^{(i)} : 
   (t_j^{(i)})_{1 \le j \le r_i} \in C_+^{(i)}
   \quad \text{for all }i\,\,(1 \le i \le L)\}.
\end{equation}

\smallskip
By using the Killing form, we identify $\sqrt{-1} \g$ with 
$\sqrt{-1} \g^*$, and regard 
$$
\sqrt{-1} \c( \k)^* \subset \sqrt{-1} (\t^{\tau})^{*} 
\subset \sqrt{-1} \t^* \subset \sqrt{-1} \k^* \subset \sqrt{-1} \g^*
$$ 
corresponding to the inclusion 
$\c(\k) \subset \t^{\tau} \subset  \t \subset \k \subset \g$. Via the 
identification $\sqrt{-1} \g^* \simeq \sqrt{-1} \g$, the condition 
\eqref{E:OGZ} for a coadjoint orbit $\mathcal{O}^G$ is equivalent to the 
condition 
$$
\mathcal{O}^G \cap \sqrt{-1} \c(\k) \ne \emptyset
$$ 
for an adjoint orbit $\mathcal{O}^G$
 (by abuse of notation)
 because $\sqrt{-1}([\k,\k] + \p)^{\perp} = \sqrt{-1} \c(\k)^{\ast}$. 

\smallskip
A coadjoint orbit $\mathcal{O}^G $ is said to be an {\it elliptic orbit}
 if 
$\mathcal{O}^G \cap \sqrt{-1} \k^* \ne \emptyset$. 
In particular,
 $\mathcal{O}^G$ is elliptic
 if \eqref{E:OGZ} is satisfied.  
If $\mathcal{O}^G$ is an elliptic coadjoint orbit,
 then $\mathcal{O}^G $ 
meets at a single point, say $\mu$, in the dominant Weyl chamber 
$\sqrt{-1} (\t^*)_{+}$ with 
respect to $\Delta^{+} (\k, \, \t)$. 
We shall write $\mathcal{O}^{G}_{\mu}$ for $\mathcal{O}^G$
 if $\mathcal{O}^G \cap \sqrt{-1} (\t^{\ast})_+=\{\mu\}$. 
Likewise, an elliptic coadjoint orbit $\mathcal{O}^H$ 
is written as $\mathcal{O}^{H}_{\mu} = \Ad^{*}(H) \mu $ for some dominant 
element $\mu \in \sqrt{-1} (\t^{\tau})^{*}_{+}$.

\smallskip 
The coadjoint orbit $\mathcal{O}^G$ satisfying \eqref{E:OGZ} is a special case 
of elliptic orbits. In this case, $\mathcal{O}^G$ is of the form 
$\mathcal{O}^G_{\lambda}$ for some $\lambda \in \sqrt{-1} \c (\k)^*$. 
If $\lambda \ne 0$ then we have either 
\begin{equation} \label{E:pos} 
\langle \lambda,  \beta \rangle > 0 \text{ for any } \beta \in 
\Delta (\p_{+}) 
\end{equation}
or 
$$\langle \lambda,  \beta \rangle < 0 \text{ for any } \beta \in 
\Delta (\p_{+}).$$

Without loss of generality,
 we may and do assume that the condition \eqref{E:pos} is satisfied.

We shall see in Proposition \ref{P:p2} below
 that if $\mathcal{O}^G$ satisfies 
\eqref{E:OGZ} and $\mathcal O^{H}$ is a coadjoint orbit in 
$\sqrt{-1} \h^*$ such that 
$ n(\mathcal{O}^{G}_{\lambda}, \mathcal O^{H}) \ne 0$ then $\mathcal O^{H}$ 
must be an elliptic orbit, equivalently, $\mathcal O^{H}$ is of the form 
$\mathcal O^{H}_{\mu} = \Ad^{*} (H) \mu$ for some 
$\mu \in \sqrt{-1} (\t^{\tau})^*_{+}$.  
Then we determine elliptic coadjoint orbits ${\mathcal{O}}^H$
 with $n(\mathcal{O}_{\lambda}^G, \mathcal{O}^H) \ne 0$
 as follows:

\begin{thmalph} \label{T:elliptic} 
Let $G$ be a Hermitian Lie group,
 and $(G,H)$ a symmetric pair
 of holomorphic type.  
Suppose $\mathcal{O}^{G}_{\lambda}=\Ad^\ast(G)\lambda$
 with $\lambda$ satisfying \eqref{E:pos}.  
Then the following three conditions on 
 $\mu \in \sqrt{-1} (\t^{\tau})^*_{+}$ are equivalent:

\begin{enumerate}
\item[{\rm{(i)}}] $ n(\mathcal{O}^{G}_{\lambda}, \mathcal O^{H}_{\mu}) \ne 0$;
\item[{\rm{(ii)}}] $ n(\mathcal{O}^{G}_{\lambda}, \mathcal O^{H}_{\mu}) = 1$;
\item[{\rm{(iii)}}] $\mu \in \lambda + \operatorname{Cone} (\p_{+}^{-\tau})$.
\end{enumerate}
\end{thmalph}

The restriction of the projection 
 $\PR \colon \sqrt{-1} \g^{\ast} \to \sqrt{-1} \h^{\ast}$
 to a coadjoint orbit ${\mathcal{O}}^G$ is identified
 with the momentum map
 $\mu \colon {\mathcal{O}}^G \to \sqrt{-1} \h^{\ast}$
 for the Hamiltonian action
 on the symplectic manifold ${\mathcal{O}}^G$.  
Then the following corollary is deduced
 readily from Theorem \ref{T:elliptic}.  
\begin{cor}
\label{cor:proper}
Let $G$ be a Hermitian Lie group, 
 and $(G,H)$ a symmetric pair of holomorphic type.  
Suppose ${\mathcal{O}}^G$ is a coadjoint orbit
 satisfying \eqref{E:pos}.  
Then the momentum map
\[
  \mu \colon {\mathcal{O}}^G \to \sqrt{-1} \h^{\ast}
\]
is proper.  
\end{cor}

The representation-theoretic counterpart for Theorem \ref{T:elliptic}
 is branching laws of scalar holomorphic discrete series representations $\pi^G(\lambda)$ 
 with respect to symmetric pairs $(G, H)$
 of holomorphic type, 
 and that for Corollary \ref{cor:proper} is discrete decomposability
 of the restriction $\pi^G(\lambda)|_H$. 
To describe the branching law explicitly, 
 we fix some notation.  
A holomorphic discrete series representation of $G$ 
 is parametrized by its minimal $K$-type.  
We denote by $\pi^G(\lambda)$ 
 if its minimal $K$-type has highest weight $\lambda$
 with respect to $\Delta^+(\k_{\mathbb{C}}, \t_{\mathbb{C}})$, 
 see also \cite[Sect.~8.1]{kobayashi07}
 for the convention
 when $G$ is not simple.  
Similar notation is applied to holomorphic discrete series representations
 $\pi^H({\mu})$ of $H$. 

\begin{fact} 
[Hua--Kostant--Schimd--Kobayashi, \cite{kobayashi97}]
\label{F:f1}
Suppose $(G, H)$ is a symmetric pair of holomorphic type. 
Assume $\lambda \in \sqrt{-1} \c( \k)^*$ satisfies \eqref{E:pos}. 
Then the restriction of $\pi^G({\lambda})$ to $H$ is decomposed 
 into a multiplicity-free direct sum of irreducible representation of $H$: 
\begin{equation}
\label{eqn:HKSK}
\pi^G({\lambda})|_{H} \simeq  
{\sum}
{}^{\raise 3pt \hbox{\kern-0.3pt$\oplus$}}
\pi^H(\lambda|_{\t^{\tau}} 
   + \sum_{i=1}^L \sum_{j=1}^{r_i}  a_j^{(i)} \nu_{j}^{(i)}), 
\end{equation}
where the sum is taken over the following countable set:
\begin{equation}
\label{eqn:aij}
  (a_j^{(i)})_{1 \le j \le r_i}\in \Lambda^{(i)}
  \quad
  (1 \le i \le L).  
\end{equation}
\end{fact}

When $H$ is a maximal compact subgroup $K$, 
 each summand in \eqref{eqn:HKSK} is finite-dimensional,
 and the formula \eqref{eqn:HKSK} was known
 by Hua \cite{hua} (classical groups), 
 Kostant (unpublished), 
and Schmid \cite{schmid}.  
The general case for noncompact $H$
 was given in Kobayashi \cite{kobayashi97} 
 with detailed proof in {\cite[Thm.~8.3]{kobayashi07}}. 
See also \cite[Cor.~3.12]{k12}
 for a formulation in the category ${\mathcal{O}}$.

\smallskip
%\noindent 
In comparison to Fact \ref{F:f1}, Theorem C may be restated as follows: 
\begin{thmalph}
\label{thm:Cprime}
Suppose $\lambda $ satisfies \eqref{E:pos}. Then $n ( \mathcal{O}_{\lambda}^G, \, \mathcal{O}_{\mu}^H) \ne 0$ 
if and only if 
$$\mu \in {\operatorname{Conv}}\bigl(\operatorname{Supp}_{H} \bigl(
\pi^G(\lambda)|_{H} \bigr)\bigr),$$ 
where 
${\operatorname{Conv}}(S)$ denotes the convex hull of a set $S$, 
 and $\mathcal{O}_{\lambda}^G := \Ad^{\ast}(G) \cdot \lambda$
 and $\mathcal{O}_{\mu}^H := \Ad^{\ast}(H) \cdot \mu$.  
\end{thmalph}

In Theorem \ref{thm:Cprime}, 
 we have regarded ${\operatorname{Supp}}_H(\pi^G(\lambda)|_H)$
 as a subset of dominant integral weights with respect to the positive system
 $\Delta^+(\k^{\tau}, \t^{\tau}) = \Delta^+(\h \cap \k, \t^{\tau})$, 
 namely, 
\begin{equation}
\label{eqn:Supp}
   {\operatorname{Supp}}_H(\pi^G(\lambda)|_H)
   =
   \bigcup
   \{\lambda|_{\t^{\tau}} 
     + \sum_{i=1}^L \sum_{j=1}^{r_i} a_j^{(i)} \nu_j^{(i)}\}, 
\end{equation}
where the union is taken over the countable set \eqref{eqn:aij}.

\medskip
\section{Proof of the main theorems}
This section gives the proof of Theorems \ref{thm:A} and \ref{T:elliptic}.  

\subsection{$(G, G^{\tau})$ and its associated symmetric pair
$(G, G^{\tau \theta})$} 

In general,
 it is not easy to describe $H$-coadjoint orbits
 on the intersection 
$$
\mathcal{O}^G \cap \PR^{-1} (\mathcal{O}^H)
$$
for a pair $(G,H)$ of reductive Lie groups.  
When $(G,H)$ is a symmetric pair, 
 our key idea is to use another symmetric pair $(G, H^a)$, 
 referred to as the {\it{associated symmetric pair}}, 
 defined as follows.  

\smallskip
Let $\tau$ be an involutive automorphism of $G$
 commuting with the Cartan 
involution $\theta$. Then, the composition $\tau \theta$ is also an 
involutive automorphism. We set 
$$H := (G^{\tau})_0, \quad H^a := (G^{\tau \theta})_0$$ 
the identity components of the fixed point groups 
$G^{\tau}$ and $G^{\tau \theta}$, respectively.
Then the reductive groups $H$ and $H^a$ have the following Cartan 
decompositions 
$$
H= (H \cap K) \exp(\p^{\tau}), \quad \! H^a= (H \cap K) \exp(\p^{-\tau}),
$$ 
respectively. 
We observe that both $H$ and $H^a$ have the same maximal 
compact subgroups $H \cap K$, and that the `noncompact part' is complementary 
to each other, namely,  
$$
\p = \p^{\tau} + \p^{-\tau} \qquad \text{(direct sum decomposition)}.
$$
This observation will be crucial 
 in the proof of Theorems \ref{thm:A} and \ref{T:elliptic} below.  

\subsection{Hermitian symmetric space $H^a / H^a \cap K$} 
We return to our previous setting where $\tau$ is of holomorphic type, 
equivalently, $\tau Z = Z$. 
Since $\theta Z = Z$,
 the involution $\tau \theta$ is also of 
holomorphic type, and consequently, $\tau \theta$ preserves the decomposition 
\eqref{E:kp}. Therefore, we have a compatible direct sum decomposition of the 
complexified Lie algebra $\g^{\tau \theta}_{\C}$ of $H^a$:
$$
\g^{\tau \theta}_{\C} = (\k^{\tau}_{\C}) + (\p^{-\tau}_{+})
+ (\p^{-\tau}_{-}).
$$
This decomposition makes $H^a / (H^a \cap K)$ a Hermitian symmetric space, 
which is naturally embedded into the Hermitian symmetric space $G/K$. 

\medskip 
We first prepare notation
 when $H^a=(G^{\tau \theta})_0$ contains only one noncompact simple factor,
 namely,
 $L=1$
 by applying the structural results
 of Hermitian symmetric spaces \cite{korany}
 to $H^a/(H^a \cap K)$.  
In this case we shall write $\{ \nu_1, \cdots, \nu_r \}$ 
 for the maximal set of strongly orthogonal roots in $\Delta(\p^{-\tau}_{+})$
 instead of $\{\nu_1^{(i)}, \cdots, \nu_{r_i}^{(i)}\}$
 as in Section \ref{SS:s2}. 
For each $j$, 
 we define an $\sL_{2}$-triple $\{H_j, E_j, E_{-j} \}$ in 
$\g^{\tau \theta}_{\C}$ as follows: 

\smallskip
$$E_j \in (\g^{\tau \theta}_{\C})_{\nu_j}, \,
E_{-j}\in (\g^{\tau \theta}_{\C})_{-\nu_j}, \, \text{ and } 
H_j \in \sqrt{-1} \t^{\tau \theta}.$$  
Here $(\g^{\tau \theta}_{\C})_{\nu_j}$ denotes the root space in 
$\g^{\tau \theta}_{\C}$ corresponding to 
$\nu_j \in \sqrt{-1} (\t^{\tau})^*$, and 
$H_j := \frac{2 \nu_j}{\langle \nu_j, \nu_j \rangle}$ 
if we identify $\sqrt{-1} \t^*$ with $\sqrt{-1} \t$ by the Killing form. 
Furthermore, we may and do choose 
$E_j$ and $E_{-j}$
 such that the following elements $X_j$ and $Y_j$ belong to the real Lie algebra $\g$: 

%\begin{align}
$$
X_j := E_j + E_{-j}, \quad Y_j := - \sqrt{-1}(E_j - E_{-j}). 
$$
%\end{align}
Then 
$
X_j, \, Y_j \in \p^{\tau \theta} = \p^{-\tau}.
$
Next, let us define the following two subspaces:

\begin{alignat}{2}
\a &:= \bigoplus_{j=1}^r \R X_{j} \quad \! 
   &&\subset 
\p^{\tau \theta} (= \p^{-\tau}), 
\label{E:r4} 
\\
\t^{-} &:= \sqrt{-1} \bigoplus_{j=1}^{r} \mathbb{R} H_j \quad \! 
    &&\subset \t^{\tau \theta} (= \t^{\tau}). 
\notag 
\end{alignat}
Let  $\t^{+}$ be the orthogonal complement of $\t^{-}$ in $\t^{\tau}$ with 
respect to the Killing form. Then $\t^{+} + \a$ is a maximally split Cartan 
subalgebra of $\g^{\tau \theta}$. 

\smallskip 
For the general case 
 where $L$ may be greater than 1, 
 we write $X_j^{(i)}$ instead of $X_j$
 ($1 \le j \le r_i$, $1 \le i \le L$).  
We take a positive system 
$\Sigma^{+}(\g^{\tau \theta}, \a)$ such that the corresponding dominant 
Weyl chamber ${\mathfrak{a}}_+$ is given by 
\begin{align}
  {\mathfrak{a}}_+^{(i)} &:=\{
    \sum_{j=1}^{r_i} t_j^{(i)} X_j^{(i)} 
                      :
                      (t_j^{(i)})_{1 \le j \le r_i} \in C_+^{(i)} \}
\quad(1 \le i \le L), 
\notag
\\
\label{E:r6} 
  {\mathfrak{a}}_+ &:=\{
    \sum_{i=1}^{L} \sum_{j=1}^{r_i}
                      t_j^{(i)} X_j^{(i)} 
                      :
                      (t_j^{(i)})_{1 \le j \le r_i} \in C_+^{(i)}
\quad
\text{for $1 \le i \le L$}\}.  
\end{align}

Correspondingly,
 we define a subset of the connected abelian group 
$A= \exp(\a)$ by 
\begin{equation}  \label{E:r7} 
A_+ = \exp(\a_{+})=\exp(\a_{+}^{(1)}) \cdots \exp(\a_{+}^{(L)}). 
\end{equation}  

Via the Killing form, 
 the projection $\PR \colon \sqrt{-1}\g^* \rightarrow  \sqrt{-1}\k^*$ is 
identified with the map  
\begin{equation}
\label{eqn:prtheta}
\PR^{\theta} : \sqrt{-1}\g \rightarrow \sqrt{-1}\k, \; \; 
X \mapsto \frac{1}{2} (X + \theta X)
\end{equation}
We recall from \cite[Prop.~2.4 and Lem.~2.5]{nasrin} an explicit 
formula for $\PR^{\theta} \bigl(\Ad(a) Z \bigr)$
 applied to each noncompact simple factor
 $G^{(i)}$ of $H^a$.

%Then we have:

\begin{lem} \label{L:l1}
Suppose $1 \le i \le L$.  
Let $Z^{(i)} \in \sqrt{-1}{\mathfrak{c}}({\mathfrak{k}}^{(i)})$ 
 be the characteristic element 
 of the simple Hermitian Lie algebra
 $\g^{(i)} = \k^{(i)} + \p^{(i)}$.  
For $t_1^{(i)}$, $\cdots$, $t_{r_i}^{(i)} \in {\mathbb{R}}$, 
 we define an element $a^{(i)}$ of $A$ by 
\[
  a^{(i)}:=\exp (\sum_{j=1}^{r_i}t_j^{(i)} X_j^{(i)}).  
\]
Then we have 
\begin{enumerate}
\item[{\rm{(1)}}] $\PR^{\theta} (\Ad(a^{(i)}) Z^{(i)}) = Z^{(i)} + \sum_{j=1}^{r_i}(\sinh t_j^{(i)})^{2} H_j^{(i)}$;
\item[{\rm{(2)}}]  $\PR^{\theta} (\Ad(a^{(i)}) Z^{(i)}) \in \sqrt{-1} (\t^{\tau})^{*}_{+}$;
\item[{\rm{(3)}}] $a^{(i)} \mapsto \PR^{\theta} (\Ad(a^{(i)}) Z_i)$ is injective when restricted to 
$A_+^{(i)}:=\exp({\mathfrak{a}}_+^{(i)})$.  
\end{enumerate}
In (2) we have identified $\sqrt{-1}\g$ with $\sqrt{-1}\g^{\ast}$
 via the Killing form.  
\end{lem}

\smallskip
\subsection{The formula for $\PR \bigl(\Ad (a) \lambda \bigr)$}

Suppose $\lambda \in \sqrt{-1} ([\k, \k] + \p)^{\perp}$. By identifying 
$\sqrt{-1}\g^*$ with $\sqrt{-1}\g$ we see that $\lambda$ is of the form 
$\lambda = c Z$ for some $c \in \R$.

\smallskip 
Similarly to the map $\PR^{\theta}\colon \g \to \k$
 (see \eqref{eqn:prtheta}),
 we define a linear map 
$$
\PR^{\tau} \colon \sqrt{-1}\g \rightarrow \sqrt{-1}\h, \; \; 
X \mapsto \frac{1}{2} (X + \tau X),
$$
which is identified with the projection 
$ 
\PR\colon \sqrt{-1}\g^* \rightarrow \sqrt{-1}\h^*
$. 
Then we have:

\begin{prop} \label{P:p1}
Suppose $\lambda = c Z$ with $c>0$. 
Recall from \eqref{E:cone} the definition of the closed cone
 ${\operatorname{Cone}}(\p_+^{-\tau})$ in $\sqrt{-1}(\t^{\tau})^{\ast}$.  
Then,  
$$\PR^{\tau}\bigl(\Ad (A_+) \lambda \bigr) = 
\PR^{\theta}\bigl(\Ad (A_+) \lambda \bigr) = \lambda + 
\operatorname{Cone} (\p_{+}^{-\tau}).$$ 
\end{prop}

\begin{proof}
We recall that 
$\a_{+} \subset \g^{-\tau} \cap \g^{-\theta} \subset \g^{\tau \theta}$. 
Since $\tau$ is of holomorphic type, we have $\tau \lambda = \lambda$, and 
therefore, $\Ad (A_+) \lambda \subset \sqrt{-1}\g^{\tau \theta}$. 
Since $\tau = \theta$ on $\g^{\tau \theta}_{\C}$, we have 

\begin{equation} \label{E:pr1}
\PR^{\tau} = \frac{1}{2}(\operatorname{id} + \tau) 
= \frac{1}{2}(\operatorname{id} + \theta) = \PR^{\theta} \quad 
\text{on $\g^{\tau \theta}_{\C}$}. 
\end{equation}

In particular,
 we have 
$$
\PR^{\tau}\bigl(\Ad (A_+) \lambda \bigr) = 
\PR^{\theta}\bigl(\Ad(A_+) \lambda \bigr).
$$
Therefore,
 we shall focus on $\PR^{\theta}(\Ad(A_+) \lambda)$ from now.  
According to the direct sum decomposition \eqref{eqn:Hdeco}, 
 the characteristic element $Z \in \c(\k)$ is decomposed as 
\[
  Z=Z^{(0)} + Z^{(1)} + \cdots + Z^{(L)}, 
\]
where $Z^{(i)} \in \sqrt{-1}{\mathfrak{c}}(\k^{(i)})$
 for $0 \le i \le L$, 
 and $Z^{(i)}$ $(1 \le i \le L)$
 are the characteristic elements
 for the Hermitian Lie algebra $\g^{(i)}$.  
We apply Lemma \ref{L:l1} (1) for the computation of  
$\PR^{\theta}\bigl(\Ad(a) \lambda \bigr)$
 with $a=a^{(1)} \cdots a^{(L)} \in A_+$, and get 

\begin{align*}
\PR^{\theta} \bigl(\Ad (A_+) \lambda \bigr)
=&
c\, \{Z^{(0)} + \sum_{i=1}^{L}(Z^{(i)} 
                  + \sum_{j=1}^{r_i} (\sinh t_j^{(i)})^2 H_j^{(i)})
   : (t_j^{(i)})_{1 \le j \le r_i} \in C_+^{(i)} \,(1 \le i \le L) \}
\\
=&c\, \{Z + \sum_{i=1}^{L} \sum_{j=1}^{r_i} (\sinh t_j^{(i)})^2 H_j^{(i)}
   : (t_j^{(i)})_{1 \le j \le r_i} \in C_+^{(i)} \,(1 \le i \le L) \}.  
\end{align*}
Hence we have proved Proposition \ref{P:p1}
\end{proof}

Next we prove the following proposition.  

\begin{prop} \label{prop:l2}
Fix $\lambda = c Z$ with $c>0$.  
Then the following three conditions on $a, a' \in A_{+}$ are equivalent:
\begin{newenumerate}
\item[{\rm{(i)}}] $\PR^{\tau}(\Ad(a) \lambda)$ and $\PR^{\tau}(\Ad(a') \lambda)$ are 
conjugate by the adjoint action of $H$;
\item[{\rm{(ii)}}] $\PR^{\tau}(\Ad(a) \lambda)$ and $\PR^{\tau}(\Ad(a') \lambda)$ are 
conjugate by the adjoint action of $H \cap K$;
\item[{\rm{(iii)}}] $a = a'$.
\end{newenumerate}
\end{prop}

\begin{proof}
Since $\PR^{\tau}\bigl(\Ad (a) \lambda \bigr) = 
\PR^{\theta}\bigl(\Ad(a) \lambda \bigr)$ for any $a \in A$,
 we see
 that $\PR^{\tau}(\Ad(a) \lambda) \in \sqrt{-1}(\t^{\tau})_+^{\ast}$
{}from Lemma \ref{L:l1} (2).   
Since two elements in $\sqrt{-1}(\t^{\tau})_+^{\ast}$ 
 is conjugate under $H=(G^{\tau})_0$
 if and only if they coincide,
 we get the implications (i) $\Rightarrow$ (ii) $\Rightarrow$ (iii)
 by Lemma \ref{L:l1} (3).  
The implication (iii) $\Rightarrow$ (i) is obvious.  
Thus Proposition \ref{prop:l2} is proved.  
\end{proof}

\medskip
\subsection{Proof of Theorems A and C}

\smallskip
Since $\a$ is a maximal abelian subspace of $\g^{- \tau} \cap \g^{- \theta}$, 
we have the generalized Cartan decomposition \cite[Thm.~4.1]{flen-jen} 
 for the semisimple symmetric pair $(G,H)$:

\begin{equation} \label{E:car-decomp}
G = H A_+ K. 
\end{equation}

\medskip
Suppose $\lambda \in \sqrt{-1} \c (\k)^*$. Since $K$ stabilizes $\lambda$, 
the decomposition \eqref{E:car-decomp} implies
\begin{equation} \label{E:adorbit}
\mathcal{O}^{G}_{\lambda} = \Ad(G) \lambda = \Ad (H) \Ad (A_{+}) \lambda. 
\end{equation}

\noindent
%{\bf \underline{Proof of Theorem A.}} \quad \!
\begin{proof}
[Proof of Theorem \ref{thm:A}]
We take any two elements $x,\, x' \in \mathcal O^{G} \cap 
(\PR^{\tau})^{-1} (\mathcal O^{H})$. We shall prove that $x' \in \Ad(H) x.$
It follows from the generalized Cartan decomposition \eqref{E:car-decomp} that 
there exist $a, \, a' \in A_{+}$ and $h, \, h' \in H$ such that 
\begin{equation} \label{E:xhaz}
x = \Ad(h) \Ad(a) \lambda, \; x' = \Ad(h') \Ad(a') \lambda. 
\end{equation}

Since the projection $\PR^{\tau}\colon \g \rightarrow \h$ respects 
$H$-action, we have 

$$\PR^{\tau}(x) = \Ad(h) \PR^{\tau}(\Ad(a) \lambda), \;
\PR^{\tau}(x') = \Ad(h') \PR^{\tau}(\Ad(a') \lambda).$$ 
By our assumption,
 both $\PR^{\tau}(x)$ and $\PR^{\tau}(x')$ are 
contained in the same $H$-orbit $\mathcal O^{H}$. 
Therefore, 
$\PR^{\tau}(\Ad(a) \lambda)$ and $\PR^{\tau}(\Ad(a') \lambda)$ are conjugate 
by an element of $H$. 
By Proposition \ref{prop:l2},
 we conclude $a = a'$. 
Using \eqref{E:xhaz} again, we see that $x$ is conjugate to $x'$ under $H$. 
This is what we wanted to prove. 
\end{proof}

Finally,
 we shall determine $H$-coadjoint orbits ${\mathcal{O}}^H$
 such that $n({\mathcal{O}}^G, {\mathcal{O}}^H) \ne 0$.  
The first step is to show 
 that ${\mathcal{O}}^H$ must be an elliptic orbit.  

\medskip
%For the proof of Theorem \ref{T:elliptic} and \ref{T:hyper}, we shall use the 
%following: 

\begin{prop} \label{P:p2}
Let $\mathcal{O}^G$ and $\mathcal{O}^H$ be coadjoint orbits in $\sqrt{-1}\g^*$ 
and $\sqrt{-1}\h^*$, respectively. Suppose $\mathcal{O}^G$ satisfies 
\eqref{E:OGZ}.  
If $\mathcal{O}^{G} \cap \PR^{-1}(\mathcal O^{H}) \ne \emptyset$, 
 then $\mathcal O^{H}$ is an elliptic orbit. 
\end{prop}

\begin{proof}
If $\mathcal{O}^G = \{0 \}$, 
 then the condition 
$\mathcal{O}^{G} \cap \PR^{-1}(\mathcal O^{H}) \ne \emptyset$ obviously
 implies 
$\mathcal{O}^H = \{0 \}$. Hence $\mathcal{O}^H$ is an elliptic orbit. 

From now, we assume that $\mathcal{O}^G \ne \{0 \}$. Without loss of 
generality, we may assume $\mathcal{O}^G = \Ad (G) \lambda$ where 
$\lambda = c Z \quad \! (c>0)$
 via the identification of $\sqrt{-1} \g$ with $\sqrt{-1} \g^*$ as before. 
If $\mathcal{O}^{G} \cap \PR^{-1}(\mathcal O^{H}) \ne \emptyset$, we find 
$g \in G$ such that 
\begin{equation}
\label{eqn:prgH}
\PR \bigl(\Ad(g) \lambda \bigr) \in \mathcal O^{H}.
\end{equation}
We write 
$$
g = h a k \in G \quad (h \in H, \, a \in A_+, \, k \in K)
$$ 
according to \eqref{E:car-decomp}.
Then it follows from \eqref{eqn:prgH}
 that $\PR \bigl(\Ad(a) \lambda \bigr) \in \mathcal O^{H}$ 
 because $\Ad(k)\lambda=\lambda$.  
By Proposition \ref{P:p1}, we have 
$\PR \bigl(\Ad(a) \lambda \bigr) \in \sqrt{-1} (\t^{\tau})^{*}_{+}$. Hence 
$\sqrt{-1} (\t^{\tau})^{*}_{+} \cap \mathcal O^{H} \ne \emptyset .$ 
Therefore, $\mathcal O^{H}$ must be an elliptic orbit.
\end{proof}

\smallskip
\noindent 
%%%{\bf \underline{Proof of Theorem C.}} \quad \! Suppose 
{\textit{Proof of Theorem C.}} \enspace 
Suppose 
$\mathcal{O}^{G} = \mathcal{O}^{G}_\lambda$ with 
$\lambda = c Z \quad \! (c>0)$. Then the proof of Proposition \ref{P:p2} 
asserts that if 
$\mathcal{O}^{G}_\lambda \cap \PR^{-1}(\mathcal O^{H}) \ne \emptyset$, 
then $\PR \bigl(\Ad(a) \lambda \bigr) \in \mathcal O^{H}$ for some 
$a \in A_+$. Clearly, the opposite implication also holds. Thus we have 
shown that $n(\mathcal{O}^{G}_\lambda, \mathcal{O}^H) \ne 0$ if and only if 
$\mathcal{O}^H \cap \bigl( 
\lambda + \operatorname{Cone} (\p_{+}^{-\tau}) \bigr) \ne \emptyset$ because 
$$\PR \bigl(\Ad(A_+) \lambda \bigr) = 
\lambda + \operatorname{Cone} (\p_{+}^{-\tau})$$ 
by Proposition \ref{P:p2}. 
Hence we have the equivalence 
(i) $\Leftrightarrow$ (iii) in Theorem C. The equivalence 
(i) $\Leftrightarrow$ (ii) follows from Theorem A.   \qed.

\section{Visible actions on coadjoint orbits}
We end this article with discussion about another aspect
 on the geometry of the coadjoint orbits.

A holomorphic action of a Lie group $H$ on a connected complex manifold $M$
 is said to be {\it{strongly visible}}
 if there exist a totally real submanifold $S$, 
 referred to as a {\it{slice}}, 
 and an anti-holomorphic diffeomorphism $\sigma$ of $M$
 which preserves every $H$-orbit in $M$
 such that generic $H$-orbits meet $S$
 and $\sigma|_S = {\operatorname{id}}$, 
 see \cite[Def.~3.3.1]{kobayashi05}.  
The proof of the multiplicity-free theorem
 (Fact \ref{F:mf}) is based
 on the following fact:
\begin{fact}
[{\cite{kobayashi07TG}}]
\label{fact:visible}
Let $G$ be a Hermitian Lie group.  
For any symmetric pair $(G,H)$, 
 the $H$-action on $G/K$ is strongly visible.  
\end{fact}
Any nonzero coadjoint orbit ${\mathcal{O}}^G$
 satisfying the condition \eqref{E:OGZ}
 is isomorphic to the Hermitian symmetric space $G/K$.  
Hence Fact \ref{fact:visible} may be seen
 as a result on the geometry of coadjoint orbits:
\begin{fact}
\label{fact:OGvisible}
Let ${\mathcal{O}}^G$ be a coadjoint orbit 
satisfying \eqref{E:OGZ}.  
For any symmetric pair $(G,H)$, 
 the $H$-action on the coadjoint orbit ${\mathcal{O}}^G$
 is strongly visible.  
\end{fact}
In this case,
 the slice $S$ can be taken to be $\overset\circ{A}_+ \cdot o$, 
 where $\overset\circ{A}_+$ denotes 
 the set of interior points of $A_+$
 and $o$ is the fixed point of $K$ 
 with notation for the generalized Cartan decomposition \eqref{E:car-decomp}.  

In turn, 
 the Cayley transform of $A_+$ for the subgroup $H^a=(G^{\tau\theta})_0$
 (not for $H=(G^{\tau})_0$)
 followed by the shift of $\lambda|_{\t^{\tau}}$
 gives the support
 ${\operatorname{Supp}}_H(\pi^G|_H)$, 
 as is seen in \eqref{eqn:Supp}, 
 where $\pi^G$ is the irreducible unitary representation of $G$
 attached to the coadjoint orbit ${\mathcal{O}}^G$
 and $\lambda$ is determined by ${\mathcal{O}}^G$
 by the condition
$
   {\mathcal{O}}^G \cap \sqrt{-1} \c (\k)^{\ast} =\{\lambda\}.  
$
This viewpoint from visible actions provides yet another perspective
 of Theorem \ref{thm:Cprime}
 on the Kirillov correspondence
 between branching laws of unitary representations
 and coadjoint orbits with momentum map
 $\mu \colon {\mathcal{O}}^G \to \sqrt{-1} \h^{\ast}$
 for the Hamiltonian action
 of the subgroup $H$ on ${\mathcal{O}}^G$.  

\vskip 3pc
{\bf{Acknowledgements:}}\enspace
The author would like to express his gratitude to the organizers,
 A.~Borodin, A.~Kirillov, S.~Morier-Genoud, A.~Okounkov, V.~Ovsienko, 
M.~Pevzner, N.~Rozhkovskaya, M.~Schlichenmaier, and R.~Yu,
of  the conference
\lq{Representation Theory
 at the Crossroads of Modern Mathematics}\rq\
 in honor of Alexandre A.~Kirillov on his 81st birthday
at Reims, May 29--June 2, 2017 
for their warm hospitality
during the stimulating conference.

This work was partially supported by Grant-in-Aid for Scientific 
Research (A)
(18H03669), Japan Society for the Promotion of Science.


\medskip
\begin{thebibliography}{99}
\bibitem{corwin}{\sc Corwin, L. and Greenleaf, F.}, {\it Spectrum and 
multiplicities for restrictions of unitary representations in nilpotent Lie groups}, Pacific J.~Math.~{\bf 135} (1988), 233--267.

\bibitem{xDeltour2013}
{\sc Deltour, G.}, 
{\it{On a generalization of a theorem of McDuff}}, 
J. Differential Geom. {\bf{93}} (2013), 379--400. 

\bibitem{duflovargas}
{\sc Duflo, M. and Vargas, J.~A.},  
\textit{Branching laws for square integrable representations},
Proc.~Japan Acad.~Ser.~A,
Math.~Sci., 
\textbf{86},
(2010),
49--54.  

\bibitem{flen-jen}{\sc Flensted-Jensen, M.}, 
{\it Spherical functions of a real 
semisimple Lie group. A method of reduction to the complex case}, 
J.~Funct.~Anal.~{\bf 30} (1978), 106--146.

\bibitem{fujiwara}{\sc Fujiwara, H.}, 
{\it Repr\'esentation monomiales des groups de Lie r\'esolubles exponentiels}, 
The Orbit Method in Representation Theory, Progr.~Math.~Birkh\"auser, 1990, 
61--84.

\bibitem{xHKM}
{\sc Hilgert, J. Kobayashi, T. and M\"ollers, J.}, 
{\it Minimal representations via Bessel operators}, 
J.~Math.~Soc. Japan {\bf{66}} (2014), 
\href{http://dx.doi.org/10.2969/jmsj/06620349}{349--414}.  

\bibitem{hua}{\sc{Hua}, L.-K.}, 
{\it Harmonic analysis of functions of several complex variables in the classical domains}, 
Translated from the Russian by Leo Ebner and Adam Kor{\'a}nyi, Amer.~Math.~Soc., Providence, R.~I., 1963.

\bibitem{kirillov62}
{\sc Kirillov, A.~A.},
{\it{Unitary representations of nilpotent Lie groups}}, 
Uspehi Mat.~Nauk {\bf{17}} 1962, % no. 4 (106),
57--110.

\bibitem{kirillov99}
{\sc Kirillov, A.~A.}, 
{\it{Merits and demerits of the orbit method}}. 
Bull.~Amer.~Math.~Soc. (N.S.) 
{\bf{36}} (1999),  433--488.

\bibitem{kirillov}
{\sc Kirillov, A.~A.}, 
{\it Lectures on the Orbit Method}, 
Graduate Studies in Mathematics, {\bf 64}, 
Amer.~Math.~Soc.~2004. 

\bibitem{K92}
{\sc Kobayashi, T.}, 
{\it Singular Unitary Representations and Discrete Series for 
Indefinite Stiefel Manifolds $U(p,q;{\mathbb F})/U(p-m,q;{\mathbb F})$},  
Mem.~Amer.~Math.~Soc., 
\href{http://www.ams.org/bookstore-getitem/item=MEMO-95-462}{\bf{462}} 
Amer.~Math.~Soc., 1992. v+106 pages.



\bibitem{kobayashi94}
{\sc Kobayashi, T.}, 
{\it Discrete decomposability of 
the restrictions of $A_q(\lambda)$ with respect to reductive subgroups and its applications}, 
%Part I, 
Invent.~Math.~{\bf 117} (1994), 
\href{http://dx.doi.org/10.1007/s002220050203}{181--205}.   
%Part II, Ann.~of Math.~{\bf 147} (1998), 1--21; 
%Part III, Invent.~Math.~{\bf 131} (1998), 229--256.

\bibitem{kobayashi97} 
{\sc Kobayashi, T.}, {\it Multiplicity-free branching laws for 
unitary highest weight modules}, Proceedings of the Symposium on 
Representation Theory held at Saga, Kyushu 1997 (K.~Mimachi, ed.), 
(1997), 
\href{http://dml.ms.u-tokyo.ac.jp/PSRT/PSRT_18/PSRT_18_houkoku_009-017.pdf}{9--17}.

\bibitem{kobayashi94*} 
{\sc Kobayashi, T.}, {\it Harmonic analysis on homogeneous manifolds of reductive type and unitary representation theory}, 
%Sugaku \, {\bf 46} (1994), Math.~Soc. Japan (in Japanese), \, 123--143; 
Translation, Series II, Selected Papers on Harmonic Analysis, Groups, 
and Invariants (K.~Nomizu, ed.),  Amer.~Math.~Soc.~{\bf 183} (1998), 
\href{http://www.ms.u-tokyo.ac.jp/~toshi/pub/43.html}{1--31}.

\bibitem{xkAnn98}
T.~Kobayashi, 
{\textit{Discrete decomposability of the restriction of
             $A_{\frak q}(\lambda)$
            with respect to reductive subgroups {\rm{II}}---micro-local analysis and asymptotic $K$-support}}, 
Ann. of Math., 
{\bf {147}}, 
(1998), 
\href{http://dx.doi.org/10.2307/120963}
{709--729}.  


\bibitem{kobayashi05} 
{\sc Kobayashi, T.}, 
{\it  Multiplicity-free representations and visible action on complex manifolds}, Publ.~Res.~Inst.~Math.~Sci.~{\bf 41} (2005),
\href{http://dx.doi.org/10.2977/prims/1145475221}{497--549},
 an special issue commemorating the fortieth anniversary
 of the founding of RIMS.


\bibitem{kobayashi07} 
{\sc Kobayashi, T.}, 
{\it Multiplicity-free theorems of the restrictions of unitary highest weight modules with respect to reductive symmetric pairs}, In Representation Theory and Automorphic Forms, 
Progr.~Math.~{\bf 255}, Birkh\"auser 2007, \href{http://dx.doi.org/10.1007/978-0-8176-4646-2_3}{45--109}. 

\bibitem{kobayashi07TG} 
{\sc Kobayashi, T.}, 
{\it Visible actions on symmetric spaces}, 
Transform.~Groups, {\bf{12}} (2007), \href{http://dx.doi.org/10.1007/s00031-007-0057-4}{671--694}.  

\bibitem{kobayashi11}
{\sc Kobayashi, T.}
{\it Geometric quantization, limits, and restrictions---some 
examples for elliptic and nilpotent orbits}, 
In:Geometric Quantization in 
the Non-compact Setting (eds. L. Jeffrey, X. Ma and M. Vergne), 
Oberwolfach Reports Volume 8, Issue 1, 2011, European Mathematical 
Society, Publishing House. 
\href{http://dx.doi.org/10.4171/OWR/2011/09}{466--469}.  


\bibitem{k12} 
{\sc Kobayashi, T.}, 
{\it Restrictions of generalized Verma modules to symmetric pairs}, 
Transform.~Group, {\bf{17}} (2012), \href{http://dx.doi.org/10.1007/s00031-012-9180-y}{523--546}.

\bibitem{k-n03}
{\sc Kobayashi, T. and Nasrin, S.}, 
{\it Multiplicity one theorem in the orbit method}, 
Lie Groups and Symmetric Spaces: 
In memory of Professor F.~I.~Karpelevi\v{c} (ed. S.~Gindikin), 
Transl.~Series 2,  Amer.~Math.~Soc., {\bf 210} (2003), 
\href{http://www.ms.u-tokyo.ac.jp/~toshi/pub/73.html}{161--169}. 

\bibitem{k-o98}
{\sc Kobayashi, T. and \O rsted, B.}, {\it Conformal geometry and branching laws for 
unitary representations attached to minimal nilpotent orbits}, 
C.~R.~Acad.~Sci.~Paris {\bf 326} (1998), \href{http://www.ms.u-tokyo.ac.jp/~toshi/pub/73.html}{925--930}.

\bibitem{korany}
{\sc Kor{\'a}nyi, A. and Wolf, J.~A.}, {\it Realization 
of Hermitian symmetric spaces as generalized half-planes}, 
Ann.~Math.~{\bf 81} 
(1965), 265--288. 

\bibitem{xMcDuff1988}
{\sc McDuff, D.},  
{\it{The symplectic structure of K\"ahler manifolds of nonpositive curvature}}, 
J. Differential Geom. {\bf{28}} (1988), 467--475. 

\bibitem{xNasrin2003}
{\sc Nasrin, S.}
{\it Corwin--Greenleaf multiplicity function for Hermitian Lie 
groups},
 thesis, The University of Tokyo, 2003.

\bibitem{xNasrin2008}
{\sc Nasrin, S.}, 
{\it{Corwin--Greenleaf multiplicity functions 
for Hermitian symmetric spaces}}.
 Proc. Japan Acad. Ser. A Math. Sci. {\bf{84}} (2008), 
 97--100. 

\bibitem{nasrin} {\sc Nasrin, S.}, {\it Corwin--Greenleaf multiplicity 
function for Hermitian symmetric spaces and multiplicity-one theorem 
in the orbit method}, 
Internat.~J.~Math.~{\bf{21}} (2010), 279--296.  

\bibitem{xParadan2008}
{\sc{Paradan, P.-E.}, 
{\it Multiplicities of the discrete series}},
preprint, 
\href{https://arxiv.org/abs/0812.0059}{arXiv:0812.0059}.  

\bibitem{xParadan2015}
{\sc Paradan, P.-E.},
{\it{Quantization commutes with reduction in the non-compact setting: 
the case of holomorphic discrete series}},
 J.~Eur.~Math.~Soc. (JEMS) {\bf{17}} (2015), no. 4, 955--990.

\bibitem{schmid}{\sc Schmid, W.}, {\it Die Randwerte holomorpher Funktionen auf hermitesch symmetrischen 
R\"aumen}, Invent.~Math.~{\bf 9} (1969/1970), 61-- 80. 

\bibitem{sj}{\sc Sjamaar, R.}, 
{\it Convexity properties of the moment mapping, re-examined}, 
Adv.~Math.~{\bf 138} 
(1998), 46--91.

\bibitem{Vog88}
{\sc Vogan, David A., Jr.}, 
{\it Irreducibility of discrete series representations for semisimple 
symmetric spaces}, 
191--221, Adv.~Stud.~Pure Math., {\bf{14}}, 1988.
\end{thebibliography}
\end{document}